\documentclass[12pt]{amsart}

\usepackage{amsmath}
\usepackage{amssymb}  
\usepackage{latexsym} 
\usepackage{comment}
\usepackage{url}
\usepackage{fullpage,url,amssymb,amsmath,amsthm,amsfonts,mathrsfs}
\usepackage[usenames,dvipsnames]{color}
\usepackage[pagebackref = true, colorlinks = true, linkcolor = blue, citecolor = Green]{hyperref}
\usepackage[alphabetic,lite]{amsrefs}
\usepackage{enumitem}
\usepackage{amscd}   
\usepackage[all, cmtip]{xy} 
\usepackage{xfrac}
\usepackage[T1]{fontenc}
\usepackage{subfiles}
\usepackage{colonequals}

\usepackage[all]{xy}
\usepackage{fullpage}

\usepackage{color} 

\def\act#1#2%
  {\mathop{}%
   \mathopen{\vphantom{#2}}^{#1}%
   \kern-3\scriptspace%
   #2}

\newcommand{\F}{{\mathbb F}}

\newtheorem{Theorem}{Theorem}[section]
\newtheorem{Lemma}[Theorem]{Lemma}

\newtheorem{Corollary}[Theorem]{Corollary}
\newtheorem{Definition}[Theorem]{Definition}

\newtheorem{Remark}[Theorem]{Remark}

\begin{document}

\title{Algebraic capsets}

\author{Cassie Grace}
\address{School of Mathematics and Statistics, University of Canterbury, Private Bag 4800, Christchurch 8140, New Zealand}
\email{cassie.grace@pg.canterbury.ac.nz}

\author{Jos\'e Felipe Voloch}
\address{School of Mathematics and Statistics, University of Canterbury, Private Bag 4800, Christchurch 8140, New Zealand}
\email{felipe.voloch@canterbury.ac.nz}
\urladdr{http://www.math.canterbury.ac.nz/\~{}f.voloch}


\begin{abstract}
Capsets are subsets of $\F_3^n$ with no three points on a line and a capset is complete if it is not a subset of a larger capset. We study some new constructions of capsets via algebraic equations over extensions of $\F_3$. In particular we construct the smallest known complete capsets with size proportional to the best known lower bound.
\end{abstract}

\maketitle
%

\section{Introduction}

A capset is a subset $C \subset \F_3^n$ with no three points on a line. A question that has attracted a lot of interest (see e.g. \cite{Tao},\cite{Matrix}) is the following: What is the largest size of a capset in $\F_3^n$ as a function of $n$? If we denote this value by $a(n)$ then it can be shown (see \cite[Proposition 2.2]{Tyrrell}) that $c=\lim a(n)^{1/n}$ exists. Moreover, it is known that $2.2202 \le c$ (\cite{Nature}) and $c \le 2.756$ (\cite{Jordan}). The exact value of $a(n)$ is known for $n \le 6$ (\cite{capsetKnownSizes1, capetKnownSizes2, capsetKnownSizes3, capsetKnownSizes4}). It is also known that $a(8)\geq 512$ (\cite{Nature}).

\begin{Definition}
    A \emph{capset} is a subset of $\mathbb{F}_3^n$ containing no three distinct points on a line. A capset is \emph{complete} if it is not a subset of a larger capset. 
    \end{Definition}

The main purpose of this paper is to exhibit some constructions of complete capsets and, in particular, construct the smallest known complete capsets.

If $q=3^m$ is a power of $3$, choosing a basis for $\F_q/\F_3$ induces an
identification of $\F_q^d$ with $\F_3^{dm}$ for any integer $d$. We will
define some subsets of $\F_q^d$ by algebraic equations and view them as
capsets in $\F_3^{dm}$ to study their properties.

We can also consider subsets of $\F_q^d$ with no three points on an $\F_q$-line.
We call such a set a cap (to distinguish it from a capset under the above 
identification). It is clear that a cap is a capset but not conversely.
We can also consider the analogous notion of complete caps. A cap $C$ is complete as a cap if $C$ is not contained in a larger cap. That is, for every
point $P$ in the complement of $C$, there exists an $\F_q$-line through $P$
which meets $C$ in two points. We remark that a cap can be complete as a cap and yet not be complete as a capset. For example, a conic in the plane $\F_q^2, q >3$. This follows from the lemma below.

\begin{Lemma}
\label{lem:lower}
If there exists a complete capset in $\F_3^n$ with $N$ points, then
$N(N+1)/2 \ge 3^n$.
\end{Lemma}

\begin{proof}
The $\F_3$ lines through the $N(N-1)/2$ pairs of points of the capset have their third point in the complement of the capset, so if $N(N-1)/2 < 3^n -N$, there is a point in the complement not on any line through two points in the capset, hence the capset is not complete. 
\end{proof}
In \cite{Csajbok}, subsets of $\F_p^n$ with $p$ prime no three points in arithmetic progression are studied. This condition is equivalent to the capset condition when $p=3$. They prove a lower bound similar, but slightly weaker, to Lemma \ref{lem:lower} in this more general situation. They also ask if this lower bound is best possible up to a multiplicative constant (loc. cit., Problem 5.1). We give a positive answer to this question in the case $p=3$ (see Corollary \ref{cor:upper}).

\subsection*{Acknowledgements}
The second author was supported by the Marsden Fund administered by the Royal
Society of New Zealand. The authors like to thank A. Bishnoi, B. Csajb\'ok,
 J. Sheekey and G. Van de Voorde for helpful comments.

\section{Constructions}

\begin{Theorem}
\label{thm:two-conics}
Let $q=3^m$ and $C \subset \F_q^2$ be given by
$$C = \{(x,x^2) \mid x \in \F_q, x \ne 0\}\cup \{(x,-x^2) \mid x \in \F_q, x \ne 0\}.$$

Then $C$ is a capset of size $2(q-1)$. If $m$ is odd, then $C$ is complete.
\end{Theorem}

\begin{proof}
The sets $C_{\pm} = \{(x,\pm x^2) \mid x \in \F_q^*\}$ are subsets of conics and, therefore, caps in the plane so, in particular, capsets. We prove first that $C = C_+ \cup C_-$ is a cap. Suppose not and let $P,Q,R$ be distinct $\F_3$-collinear points of $C$. It is not possible that all three of $P,Q,R$ are in $C_+$ or $C_-$. We first consider the case
$P = (x_1,x_1^2), Q = (x_2,x_2^2) \in C_+, R =(x_3,-x_3^2) \in C_-$.
We have $x_1+x_2+x_3 = x_1^2+x_2^2 -x_3^2=0$. So
$x_1^2+x_2^2 = x_3^2 = (x_1+x_2)^2$, which yields $x_1x_2=0$ which is not
possible since $C$ does not contain the origin, by hypothesis.
Now, the case 
$P = (x_1,-x_1^2), Q = (x_2,-x_2^2) \in C_-, R =(x_3,x_3^2) \in C_+$
is similar to the first case. So $C$ is a cap.

Now we show that $C$ is complete. Assume that $(a,b) \notin C$. Consider first the case $a \ne 0$. Then we take $x_1 = (a^2-b)/a$ and note that $x_1 \ne 0$ since $(a,b) \notin C$. Let $x_2 = -(a+x_1)$ and we check that
$(a,b), (x_1,x_1^2), (x_2,-x_2^2)$ are $\F_3$-collinear. Indeed,
$$a+x_1+x_2=0, b+x_1^2-x_2^2 = b+x_1^2 -(a+x_1)^2= b-a^2 +ax_1=0.$$
Now, $(x_1,x_1^2) \in C$ by construction but we need to check that 
$(x_2,-x_2^2) \in C$, i.e., that $x_2 \ne 0$. If $x_2 = 0$, then $x_1=-a$
so $a^2-b = -a^2$ so $b = -a^2$, i.e. $(a,b) \in C$, contrary
to hypothesis.

Assume now that $a=0$. If $b=0$, then $(1,1),(-1,-1),(0,0)$ are $\F_3$-collinear and $(1,1),(-1,-1) \in C$. If $b=x^2, x\ne 0$, then
$(x,x^2),(-x,x^2),(0,b)$ are $\F_3$-collinear
 and $(x,x^2),(-x,x^2) \in C$. The remaining case is when $b$ is a nonsquare but, since we assumed that $m$ is odd, $-1$ is a non-square and we can write $b=-x^2, x\ne 0$ and, in this case $(x,-x^2),(-x,-x^2),(0,b)$ are $\F_3$-collinear
 and $(x,-x^2),(-x,-x^2) \in C$.

\end{proof}

The next result provides a positive answer to Problem 5.1 of \cite{Csajbok} for $p=3$. It was pointed out to us by A. Bishnoi that the construction from \cite{07731270} can also be used to deduce this result.

\begin{Corollary}
\label{cor:upper}
For every $n$, there exists a complete capset in $\F_3^n$ of size $O(\sqrt{3^n})$. 
\end{Corollary}

\begin{proof}
Assume first that $n=2m$ even and start with the construction given by Theorem \ref{thm:two-conics} of a capset of size $2(3^m -1) = O(\sqrt{3^n})$ in $\F_{3^m}^2$. If $m$ is odd, this is already complete.
When $m$ is even, we can only add points of the form $(0,b), b$ non-square. So we need a complete capset in $\F_{3^m}$ comprised of non-squares. Viewing $\F_{3^m}$ as $\F_{3^{m/2}}(\sqrt{d})$ and selecting $\lambda$ a non-square in $\F_{3^m}$, we can take elements of the form $\lambda(\pm 1+x\sqrt{d})^2, x \in \F_{3^{m/2}}^*$ which will be a capset as in the theorem, up to an affine transformation, since $(\pm 1+x\sqrt{d})^2 = (1+dx^2) \pm 2x\sqrt{d}$. Iterating this construction gives the result for $n$ even.

To construct small complete capsets in $\F_{3^n}$ when $n$ is odd, write $n=2m+1$, take the complete capset $C \subset \F_{3^{2m}}$ constructed above and consider $C \times \{0,1\} \subset \F_{3^n}$. It is straightfoward that this is a capset. To see it is complete, consider a pointi $P$ in its complement with last coordinate $c \in \F_3$ and let $P_0 \in \F_{3^{2m}}$ be the point corresponding to the other coordinates. If $c \in \{0,1\}$, then $P_0 \not\in C$ and there is a line in the subspace of points whose last coordinate is $c$ through $P$ meeting $C \times \{c\}$ in two points. If $c=2$ and $P_0 \in C$, the vertical line through $P$ meets the capset in two points. If $c=2$ and $P_0 \not\in C$, there is a line through $P_0$ meeting $C$ in $Q_0,Q_1$, say. Then $(Q_0,0), (Q_1,1)$ are in $C \times \{0,1\}$ and are collinear with $P$.
\end{proof}

We will now consider the question of whether we can construct capsets with more than two parabolas, in the spirit of Theorem \ref{thm:two-conics}. We begin with the following lemma,

\begin{Lemma}\label{lemma-conditions}
If $P_i = (x_i,c_ix_i^2), i=1,2,3$ with $x_i,c_i \in \F_{3^m}$ nonzero and $P_i$ distinct
are such that $P_1+P_2+P_3=0$, then
$-(c_1c_2+c_1c_3+c_2c_3)$ is a square in $\F_{3^m}$.
\end{Lemma}

\begin{proof}
We have $x_1+x_2+x_3 = c_1x_1^2+c_2x_2^2 +c_3x_3^2=0$. So
$c_1x_1^2+c_2x_2^2 +c_3(-x_1-x_2)^2=0$ which yields
$(c_1+c_3)x_1^2+ (c_2+c_3)x_2^2 +2c_3x_1x_2=0$. So we get a quadratic equation with root $x_1/x_2$ and so its discriminant $c_3^2-(c_1+c_3)(c_2+c_3) = -(c_1c_2+c_1c_3+c_2c_3)$ is a square.

\end{proof}

If $m$ is odd, then no choice of $a,b,c$ distinct will result in $\{(x,ax^2) \mid x \in \F_{3^m}^*\}\cup \{(x,bx^2) \mid x \in \F_{3^m}^*\}\cup \{(x,cx^2) \mid x \in \F_{3^m}^*\}$ being a capset. Indeed, applying the lemma to all possible $6$ choices of $c_i$ equal two of $a,b,c$ result in 
$$ab-a^2,ab-b^2,ac-a^2,ac-c^2,bc-b^2,bc-c^2$$
all being non-squares (since $-1$ is a non-square). If we denote by $\chi(.)$ the quadratic character of $\F_{3^m}^*$, we get using that $\chi(-1)=-1$,
$\chi(a)=-\chi(b-a)=-\chi(b)$ from the first two conditions, $\chi(a)=-\chi(c-a)=-\chi(c)$ from the next two (so $\chi(b)=\chi(c)$), but we get $\chi(b)=-\chi(c-b)=-\chi(c)$ from the last two, which is a contradiction.


For even $m$, we ask for what integers $K$ and choices of $c_i$ will $\{(x,c_i x^2) \mid c_i,x\in \F_{3^m}, \;1\leq i\leq K\}$ be a capset in $\F_3^{2m}$.

\begin{Lemma}\label{lemma-bound on conditions}
	Let $m$ be even and let $a\in \F_{3^m}$ be an element such that $a, a^3,\dots,a^{3^{m-1}}$ are distinct. Consider the set $\{(x,a^{3^i} x^2) \mid x\in \F_{3^m}, \;0\leq i\leq m-1\}$. The number of multiplicatively independent conditions of the form $-(c_1c_2+c_1c_3+c_2c_3)$ is a non-square (as follows from Lemma \ref{lemma-conditions}) is at most $\lceil (m^2+2)/6\rceil$.
\end{Lemma}
\begin{proof}
	First note that $x^{3^i} = x^{3^{i \pmod m}}$. We will represent the condition $\chi(a^{3^i}a^{3^j} + a^{3^i}a^{3^k} + a^{3^j}a^{3^k}) = -1$ as the triple $(i,j,k)$. We have that $x^3$ is a non-square if and only if $x$ is a non-square. Hence $(a^{3^i}a^{3^j} + a^{3^i}a^{3^k} + a^{3^j}a^{3^k})^3 = a^{3^{i+1}}a^{3^{j+1}} + a^{3^{i+1}}a^{3^{k+1}} + a^{3^{j+1}}a^{3^{k+1}}$ is a non-square if and only if $a^{3^i}a^{3^j} + a^{3^i}a^{3^k} + a^{3^j}a^{3^k}$ is a non-square. Therefore the condition given by $(i,j,k)$ is equivalent to the condition given by $(i+1, j+1, k+1)$ (modulo $m$).
	
	We now consider the conditions given by $(i,j,k)$, where $i< j< k$. There are $\binom{m}{3}$ such conditions, and the relation found above gives classes of $m$ conditions which are equivalent, except for when $m$ is a multiple of $3$; If $m$ is a multiple of $3$, then the condition given by $(0, m/3, 2m/3)$ is equivalent to only $m/3$ other conditions by the above relationship. Therefore the $\binom{m}{3}$ conditions can be reduced to $\binom{m}{3}/m = (m-1)(m-2)/6$ conditions if $m$ is not a multiple of three, and $(\binom{m}{3} - \frac{m}{3})/m + 1 = \left\lceil (m-1)(m-2)/6 \right\rceil$ conditions if $m$ is a multiple of three.
	
	Now consider the condition given by $(i, i, j)$ with $i\neq j$. We have seen before that this condition is equivalent to the condition given by $(0,0,j-i\pmod m)$. We now show that for all $1\leq i\leq m-1$, the conditions given by $(0,0,i)$ and $(0,0,m-i)$ are equivalent: The condition given by $(0,0,m-i)$ is equivalent to the condition given by $(i,i,0)$. This condition is $\chi(a^{2\cdot3^i} - a^{3^i+1}) = \chi(a^{3^i-1})\chi(a^{3^i+1} - a^2) = -1$. But $a^{3^i-1}$ is a square, so this is equivalent to $\chi(a^{3^i+1} - a^2) = -1$. But the condition given by $(0,0,i)$ is $\chi(a^2 - a^{3^i+1})=-1$. Hence the conditions given by $(0,0,i)$ and $(0,0,m-i)$ are equivalent. Therefore the conditions given by $(i, i, j)$, for all $i\neq j$, reduce to $m/2$ conditions.
	
	Therefore the conditions given by Lemma \ref{lemma-conditions} can be reduced to $\left\lceil (m-1)(m-2)/6\right\rceil + m/2 = \left\lceil (m^2+2)/6\right\rceil$ conditions. Hence the number of independent conditions is at most $\left\lceil (m^2+2)/6 \right\rceil$.
\end{proof}

Computationally, we have found that the bound in Lemma \ref{lemma-bound on conditions} is tight for even $m\leq 16$, and that for even $m\leq 20$ there is some choice of $a$ in $\F_{3^m}$ such that the construction in Lemma \ref{lemma-bound on conditions} is a capset. Code available at \cite{git}.

\begin{Remark}\label{remark-frobenius-construction}
    When $m$ is even, there is a capset in $\F_3^{2m}$ of the form $\{(x,a_i^{3^j} x^2) \mid a_i,x\in \F_{3^m}, \;1\leq i\leq k, \; 0\leq j\leq m-1\}$, found by computation, for the following values of $m$ and $k$ displayed in Table \ref{table-capsets1}.
\end{Remark}

\begin{table}[h!]
	\centering
	\begin{tabular}{||c c c c||} 
		\hline
		$m$ &  $k$ & number of coefficients & size of capset\\ [0.5ex] 
		\hline\hline
		2 &  1 & 2 & 16 \\ 
		\hline
		4 & 1 & 4 & 320 \\
		\hline
		6 & 1 & 6 &  4,368 \\
		\hline
		8 & 2 & 16 & 104,960 \\
		\hline
		10 & 4 & 40 & 2,361,920 \\
		\hline
		12 & 7 & 84 & 44,640,960 \\
		\hline
		14 & 11 & 154 & 736,577,072 \\ [0.5ex]
		\hline
	\end{tabular}
    \caption{} \label{table-capsets1}
\end{table}

 Table \ref{table-capsets2} shows the largest capsets in $\F_3^{2m}$ of the form $\{(x,c_i x^2) \mid c_i,x\in \F_{3^m}, \;1\leq i\leq K\}$ that we have found for small $m$.

\begin{table}[h!]
	\centering
	\begin{tabular}{||c c c c c||} 
		\hline
		$m$ & number of coefficients ($K$) & size of capset & arising from & exhaustive \\ [0.5ex] 
		\hline\hline
		2 &  2 & 16 & exhaustive search & yes \\ 
		\hline
		4 & 4 & 320 & exhaustive search & yes \\
		\hline
		6 & 8 & 5,824 & exhaustive search & yes \\
		\hline
		8 & 20 & 131,200 & random search & no \\
		\hline
		10 & 40 & 2,361,920 & Remark \ref{remark-frobenius-construction} & no \\
		\hline
		12 & 84 & 44,640,960 & Remark \ref{remark-frobenius-construction} & no \\
		\hline
		14 & 154 & 736,577,072 & Remark \ref{remark-frobenius-construction} & no \\ [0.5ex]
		\hline
	\end{tabular}
	\caption{Computational Results} \label{table-capsets2}
\end{table}


So far, we have discussed capset constructions via unions of parabolas in the plane. We can also construct capsets from paraboloids in three-space. By contrast from the plane case, these are always complete.

\begin{Theorem}
\label{thm:quadric}
Let $\lambda$ be a non-square in $\F_q, q=3^m$ and $Q \subset \F_q^3$ be given by
$$Q = \{(x,y,x^2-\lambda y^2) \mid x,y \in \F_q\}.$$
Then $Q$ is a complete capset.
\end{Theorem}

\begin{proof}
Since $Q$ is an elliptic quadric, it is a cap, so it is a capset.
To show it is complete, let $(a,b,c) \notin Q$ and consider the set
$Q' = \{(x,y,z) \mid -(a,b,c)-(x,y,z) \in Q\}$. This is clearly also
an elliptic quadric. A point $(x,y,z) \in Q \cap Q'$ gives rise to two points $(x,y,z), -(a,b,c)-(x,y,z) \in Q$ collinear with $(a,b,c)$. As algebraic varieties, $Q$ and $Q'$ meet at infinity
on a pair of (conjugate) lines so the affine part of their intersection is a (possibly singular, but it is not) conic (by B\'ezout), hence has a rational point.
\end{proof}


\end{document}